\def\misajour{\it   \rm  June 1, 2009} %

\documentclass[10pt,leqno,a4paper]{amsart}
\usepackage{amsmath,amsfonts,amssymb,amsthm,amsxtra}

\usepackage[applemac]{inputenc}

\usepackage{hyperref}

\numberwithin{equation}{section}

\newtheorem{theorem}[equation]{Theorem} 
\newtheorem{conjecture}[equation]{Conjecture} 
\newtheorem{proposition}[equation]{Proposition}

\begin{document}
\title{Perfect Powers: 
\\
Pillai's works and their developments 
}
\author{Michel Waldschmidt}
\address{Université Pierre et Marie Curie--Paris 6, 
UMR 7586 IMJ Institut de Mathématiques de Jussieu, 
175 rue du Chevaleret,
Paris, F--75013 France}
\email{miw@math.jussieu.fr \qquad http://www.math.jussieu.fr/$\sim$miw/}
\date{\misajour. { \tt Preliminary draft}}


\def\bfu{{\mathbf{u}}}
\def\bfv{{\mathbf{v}}}

\def\hfl#1#2{\smash{\mathop{\hbox to 12 mm{\rightarrowfill}}
\limits^{\scriptstyle#1}_{\scriptstyle#2}}}

 \def\hlfl#1#2{\smash{\mathop{\hbox to 12 mm{\leftarrowfill}}
\limits^{\scriptstyle#1}_{\scriptstyle#2}}}

\def\virgule{\raise 2pt \hbox{,}}

\newcommand{\m}[1]{${#1}$}

\newcommand{\M}[1]{$${#1}$$}

\def\boxit#1#2{\setbox1=\hbox{\kern#1{#2}\kern#1}%
\dimen1=\ht1 \advance\dimen1 by #1 \dimen2=\dp1 \advance\dimen2 by #1
\setbox1=\hbox{\vrule height\dimen1 depth\dimen2\box1\vrule}%
\setbox1=\vbox{\hrule\box1\hrule}%
\advance\dimen1 by .4pt \ht1=\dimen1
\advance\dimen2 by .4pt \dp1=\dimen2 \box1\relax}

\def\bC{\mathbf{C}}
\def\bF{\mathbf{F}}
\def\bN{\mathbf{N}}
\def\bQ{\mathbf{Q}}
\def\bR{\mathbf{R}}
\def\bZ{\mathbf{Z}}
 \def\bx{\mathbf {x}}
 
 \def\Qbar{\overline{\bQ}}

\def\th{\relax}
\def\tvi{\vrule height 13pt depth 8pt width 0pt}
\def\tv{\relax}
\def\cc#1{\ #1 \ }

\begin{abstract}

A {\it perfect power} is a positive integer of the form $a^x$ where $a\ge 1$   and $x\ge 2$ are rational integers.  Subbayya Sivasankaranarayana Pillai wrote several papers on these numbers. In 1936 and again in 1945  he suggested that {\it  for any given $k\ge 1$, the number of positive integer solutions $(a,\, b,\, x,\, y)$, with  $x\ge 2$ and $y\ge 2$, to the Diophantine equation $a^x-b^y=k$ is finite. }
This conjecture amounts to saying that 
{\it  the distance between two consecutive elements in the sequence of perfect powers tends to infinity}. 
After a short introduction to Pillai's work on Diophantine questions, we quote some later developments and  we discuss related open problems.

\end{abstract}

\maketitle

\tableofcontents

\section*{Acknowledgments}
Many thanks to 
R.~Thangadurai, Mike Bennett, Yann Bugeaud, Rob Tijdeman, Cam Stewart, Alain Kraus, Maurice Mignotte,  Patrice Philippon, Abderrahmane Nitaj and Claude Levesque for their comments on a preliminary version of this paper.

 \section{Pillai's contributions to Diophantine problems}\label{S:SomeRemarks}

 A large part of Pillai's work is devoted to  Diophantine questions. For instance there are deep connections with Diophantine problems in his works on Waring's problem (see \S~\ref{SS:Waring}). He was interested in Diophantine equations as early as 1930
 \cite{JFM56.0878.03}. 
Later, in 1940  \cite{JFM66.0140.02}, 
he investigated a linear Diophantine equation. He also studied irrational numbers in 
 \cite{MR0006752}. 
The paper
 \cite{MR0006753} 
 deals with a question of Diophantine approximation which was studied by Hardy and Littlewood by means of continued fractions as well as by transcendental methods, while Pillai's contribution involves only elementary arguments. 

In this brief survey, we mainly discuss questions involving perfect powers.

 \subsection{Pillai's results on Diophantine questions} \label{SS:Results}

In 1931 
 \cite{JFM57.0236.01}, 
 S.S.~Pillai 
  proved that for any fixed positive integers $a$ and $b$, both at least $2$, the number of solutions $(x,y)$ of the Diophantine inequalities $0<a^x-b^y\le c$ is asymptotically equal to 
\begin{equation}\label{E:1931a}
 \frac{(\log c)^2}{2(\log a)(\log b)}
\end{equation}
 as $c$ tends to infinity.
 It is interesting to read p.~62 of that paper: 
 
 \begin{quote}

\emph{
 The investigation was the result of the attempt to prove that the equation
\begin{equation}\label{E:1931b}
 m^x-n^y=a
 \end{equation}
 has only a finite number of integral solutions. This I conjectured some time back. 
 }
 \end{quote}

In this equation, $m$, $n$ and $a$ are fixed, the unknowns are $x$ and $y$ only. It will take some more years before he considers the same question with $m$, $n$, $x$ and $y$ as unknown, only $a$ being fixed. He continues in  \cite{JFM57.0236.01} by pointing out that the finiteness of the set of solutions $(x,y)$  to the exponential equation (\ref{E:1931b}) follows from a result of G.~P\'olya  \cite{JFM46.0240.04} , but the approach of S.S.~Pillai based on Siegel's Theorem \cite{JFM56.0180.05} provides more information. 

A chronology of the study of such problems, including references to works by C.~Störmer (1908) A.~Thue (1908), G.~P\'olya (1918), T.~Nagell (1925 and 1945), S.S.~Pillai (1945) is given by P.~Ribenboim on p.~271 of 
\cite{MR1259738}).   

In 1932
 \cite{JFM58.0187.03}, 
Pillai  proves that for $a\rightarrow\infty$, the number $N(a)$ of $(x,y)$, both positive integers $>1$, with
 $$
0< x^y - y^x \le a
$$
satisfies
$$
N(a)\sim \frac{1}{2} \frac{(\log a)^2}{(\log\log a)^2}\cdotp
$$
He deduces that for any $\varepsilon>0$ and for $a$ sufficiently large in terms of $\varepsilon$, the number of solutions to the equation  $x^y - y^x = a$ is at most
$$
(1+\varepsilon) \frac{\log a}{\log\log a}\cdotp
$$

The work started by  S.S.~Pillai  in  1931 
was pursued in  1936  by A.~Herschfeld  \cite{JFM61.1075.04,JFM62.0134.02}  who showed that if $c$ is an integer with  sufficiently large $|c|$, then the equation 
\begin{equation}\label{E:Herschfeld}
2^x-3^y=c
\end{equation}
has at most one solution $(x,y)$ in positive integers $x$ and $y$.

 For small $|c|$ this is not true. With elementary methods  Herschfeld showed that the only triples of  integers $(x,y,c)$  with positive $x$ and $y$ such that $2^x-3^y=c$  are given for $|c|\le 10$  by 
 $$
 (2,1,1),\quad
 (1,1,-1),\quad
(3,2,-1),\quad
(3,1,5),\quad
(5,3,5),\quad
(2,2,-5),\quad
(4,2,7),\quad
(1,2,-7).
$$
Thus, if $x>5$ or $y>3$, then $|2^x-3^y|>10$. By means of the same method,   Herschfeld showed that 
 if $x>8$ or $y>5$, then $|2^x-3^y|>100$ (see \cite{MR1259738}).  

 %

 S.S.~Pillai  
 \cite{JFM62.0134.01,JFM63.0113.04} 
 extended Herschfeld's result on (\ref{E:Herschfeld}) to the more general exponential Diophantine equation 
\begin{equation}\label{E:Bennett}
a^x-b^y=c,
\end{equation}
where $a$, $b$ and $c$ are fixed nonzero integers with $\gcd(a,b)=1$ and $a>b\ge 1$: he showed that there exists a positive integer $c_0(a,b)$ such that, for $|c| > c_0(a,b)$, this equation has at most one solution. 
Pillai's work (as well as Herschfeld's ones) depended on Siegel's sharpening of Thue's inequality on the rational approximation of algebraic numbers \cite{JFM56.0180.05}. The proof does not give any explicit value for $c_0(a,b)$. 

Combining this result with his previous estimate  (\ref{E:1931a}),  S.S.~Pillai  deduced that the number of integers in the range $[1,n]$ which can be expressed in the form $a^x-b^y$ is asymptotically
$$
\frac{(\log n)^2}{2(\log a)(\log b)}
$$
as $n\rightarrow\infty$. 
In the special case of Herschfeld's equation (\ref{E:Herschfeld}) with $(a,b)=(3,2)$,  S.S.~Pillai  conjectured that  $c_0(3,2)=13$, noting the equations 
$$
3-2=3^2-2^3=1,\qquad
3-2^3=3^3-2^5=-5,\qquad
3-2^4=3^5-2^8=-13. 
$$
This conjecture  (see p.~279 of  \cite{MR1259738})   was settled by  R. J. Stroeker and R. Tijdeman in 1982
\cite{MR84i:10014}
by using measures of linear independence for logarithms of algebraic numbers:
{\it if $c>13$ is a fixed positive integer, then the exponential Diophantine equation 
$$
|3^x-2^y|=c
$$ 
admits at most one solution in positive integers $x$ and $y$. }

A discussion of the work of  S.S.~Pillai  and related later works is included in Ribenboim's book    \cite{MR1259738} 
(see  pp.~ 125--127, 271--272, 279--280 and 307--308). 


The last sentence of
 \cite{JFM62.0134.01} 
is
\begin{quote}
\emph{
In conclusion, it may be that when $c$ alone is given and $X$, $x$, $Y$, $y$ are unknown, the equation
\begin{equation} \label{quote1936}
X^x-Y^y=c
\end{equation}
will have only a finite number of solutions, provided that $x\ge 2$, $y\ge 2$.
}
\end{quote} 
 In
 \cite{MR0006188}, 
 S.S.~Pillai  elaborated on  a claim by  S.~Ramanujan (see Hardy's lectures on Ramanujan 
 \cite{MR0106147}) that the number of integers of the form   $2\sp a\cdot 3\sp b$ which are bounded by a given number $x$ is approximately
  $$
 \frac{(\log (2x))(\log (3x)) }{2(\log 2)(\log 3)}\cdotp
 $$
He pursued this study in a joint paper with A.~George 
 \cite{MR0006189}. 

From Pillai's above mentioned results in 
\cite{JFM57.0236.01,
JFM62.0134.01,JFM63.0113.04},    
it follows  that for given positive integers $a$ and $b$, the equation 
$a^x-b^y=a^z-b^w$ has only a finite number of solutions $x$, $y$, $z$, $w$. In 
\cite{MR0011477}, 
he intends to prove the same result for $a^x-b^y=b^z±a^w$. However, as pointed out by A.~Brauer in his 
review MR0011477
 of \cite{MR0011477},
      this result is correct only if $\gcd(a,b)=1$, although  S.S.~Pillai  also considers the case $\gcd(a,b)>1$. This can be seen by the following examples: 
      $$ 
      3^{k+1}-3^k=3^k+3^k\quad  \text{and} \quad  2^{\alpha k+2}-(2^\alpha)^k=(2^\alpha)^k+2^{\alpha k+1}, 
      $$ 
where $a=b=3$ and $a=2$, $b=2^\alpha$, respectively.
  
 Pillai's paper
 \cite{MR0013386} 
solves completely the Diophantine equations 
$$
2^x-3^y=3^Y-2^X,\quad
2^x-3^y=2^X+3^Y\quad\hbox{and}\quad
3^y-2^x=2^X+3^Y.
$$
These solutions are displayed  on p.~280 of 
 \cite{MR1259738}. 
The number of solutions $(x,y,X,Y)$ is respectively $6$, $7$ and $8$.

 \subsection{Pillai's Conjecture on the sequence of perfect powers}\label{SS:SequencePerfectPowers}

 S.S.~Pillai  stated several conjectures on perfect powers. In 1945  \cite{MR0013386}, he wrote: 

\begin{quote}
\emph{
\null\quad
I take this opportunity to put in print a conjecture which I gave during the conference of the Indian Mathematical Society held at Aligarh.}

\par

{\it Arrange all the powers of integers like squares, cubes etc{.} in increasing order as follows:
\M{
 1,\;  4, \; 8,\; 9,\; 16,\; 25,\; 27,\; 32,\; 36,\; 49,\; 64,\; 81,\;  100,\; 121,\; 125,\; 128, \dots  
} 
Let \m{a_n} be the \m{n}-th member of this series so that \m{a_1=1}, \m{a_2=4}, \m{a_3=8}, \m{a_4=9}, etc. 
Then
\M{
\liminf_{n\rightarrow \infty } (a_n-a_{n-1})=\infty.
}}
\end{quote}
 This statement is  equivalent to the conjecture (\ref{quote1936})  in   \cite{JFM62.0134.01}:

\begin{conjecture}[Pillai's Conjecture]\label{Conj:Pillai}
For any integer $k\ge 1$, the  Diophantine equation $a^x-b^y=k$ has only finitely many positive integer solutions $(a,\, b,\, x,\, y)$, with $x\ge 2$ and $y\ge 2$.
\end{conjecture}


From the asymptotic estimate 
$$
n - a_n^{1/2}\sim a_n^{1/3}
\quad\hbox{for}\quad 
 n\rightarrow \infty
 $$
 one derives
$$
a_n =  n^2- \bigl(2+o(1)\bigr) n^{5/3}  
\quad\hbox{for}\quad 
 n\rightarrow \infty.
$$
%
Therefore 
\M{
\limsup_{n\rightarrow\infty}  (a_n-a_{n-1})=\infty.
}
In the beginning of the sequence $(a_n)_{n\ge 1}$, the smallest values for $a_n-a_{n-1}$ occur with   \m{(8,9)}, \m{(25,27)}, \m{(125,128)},  \m{(4,8)}, \m{(32,36)}, \m{(121,125)}   \m{(4,9)}, \m{(27,32)}\dots


In his 1737 paper "Variae observationes circa series infinitas", Euler attributed the formula
$$
\sum_{n\ge 2}\frac{1}{a_n-1} = 1
$$ 
 to   Goldbach.  In other words the sum of the inverse of perfect powers, excluding 1 and omitting repetitions, is $1$.
It can be shown that the sum of the inverse of perfect powers, excluding 1 and but including  repetitions, is also $1$. 
The formula
 $$  
  \sum_{k=1}^\infty \frac{1}{a_k+1}= \frac{\pi^2}{ 3} - \frac{5}{2}=0.789\, 868\, 133\, 696\, 452\, 872\, 94 \dots
  $$ 
is    quoted in \cite{SloaneEncyclopaedia}.
This reference \cite{SloaneEncyclopaedia}  includes comments by A.J.~van der Poorten, suggesting further open problems, together with the  exercise:
$$
\sum_{n=0}^{\infty}\sum_{m=1} ^{\infty}\frac{1}{(4n+3)^{2m+1}}=  \frac{\pi}{8}-\frac{1}{2}\log 2
=0.046\, 125\, 491\, 418\, 751\, 500\, 099\dots
$$
 


\section{On Pillai's Conjecture and further open problems}\label{section:PillaiConjecture}

There is only one value of $k$ for which Pillai's Conjecture \ref{Conj:Pillai}  is solved, namely $k=1$: {\it there are only finitely many pairs of consecutive integers which are perfect powers.}  This was proved by R.~Tijdeman in 1976 
\cite{MR53:7941}.  On the other hand it is known that Pillai's equation  (\ref{E:Catalan}) has only finitely many solutions when one among the four variables $x$, $y$, $m$, $n$ is fixed (Theorem 1.3 of \cite{BiluBugeaudMignotte}).

In  \S~\ref{SS:Catalan} we discuss Catalan's equation and the solution by Mih{\u{a}}ilescu of Catalan's Conjecture in 2003. Next in \S~\ref{SS:PerfectPowers} we consider further exponential Diophantine equations similar to Pillai's equation (\ref{E:Bennett}).

\subsection{Catalan's equation}\label{SS:Catalan}

For the particular value  $k=1$,  a stronger conjecture than Pillai's Conjecture \ref{Conj:Pillai} was proposed by E.~Catalan in 1844
\cite{ERAM027.0790cj}, the same year as Liouville constructed  the first examples of transcendental numbers.  Catalan suggested that the only example of consecutive integers which are perfect powers is $(8,9)$: {\it  The equation  $a_{n+1}-a_n=1$ has only one solution, namely  $n=3$. }

\begin{conjecture}[E.~Catalan]\label{C:Catalan}
The Diophantine equation $x^p-y^q=1$, where the four unknowns $(x,y,p,q)$ are integers all $>1$, has only one solution $(3,2,2,3)$.
\end{conjecture}

Special cases of the Catalan's Equation
\begin{equation}\label{E:Catalan}
x^m-y^n=1
\end{equation} 
(where the four unknowns $x$, $y$, $m$ and $n$ are integers $\ge 2$ 
) have been considered long back. References on this question are
\cite{MR58:27752,
MR88h:11002, 
MR1259738,
MR2076335, 
MR99a:11088b,
MR1761897,
MR1764814,
MR1740351,
MR2002g:11034,
MR2111637,
MR1968450,
MR2015449,
MR2152212,
HenriCohenCatalan,
MR2254063,
MR2220266,
pre05232987,
MR2459823,
BaluPandey,
BiluBugeaudMignotte}. 

The fact that the equation  \m{3^a-2^b=1} has a single solution \m{a=2}, \m{b=3}, meaning that $8$ and $9$ are the only consecutive integers in the sequence of powers of $2$ and $3$,
was a problem raised by Philippe de Vitry, French Bishop of Meaux (see \cite{JFM60.0817.03} volume II),  
and solved around 1320  by the French astronomer Levi Ben Gerson (also known as Leo Hebraeus or Gersonides, 1288--1344 -- see 
\cite{MR1259738}). 

In 1657 Frenicle de Bessy 
(see 
\cite{MR1259738}) 
solved a problem posed by Fermat: {\it if $p$ is an odd prime and $n$ an integer $\ge2$, then the equation  $x^2-1=p^n$ has no integer solution. If $n>3$, 
 the equation  $x^2-1=2^n$ has no integer solution}.

 In 1738, 
L.~Euler (see for instance \cite{MR0249355,
MR1259738,
MR2067169
})  proved that the equation is \m{x^3-y^2=\pm 1} has  no other solution than  \m{(x,y)=(2,3)}.

In 1850 V.A.~Lebesgue \cite{Lebesgue} proved that there is no solution to equation (\ref{E:Catalan}) with $n=2$.

In 1921 and 1934 (see
\cite{MR0249355,
MR2067169}) 
 T.~Nagell showed that for $p>3$, if the equation $y^2-1=y^p$ has a solution in integers with $y>0$,  then $y$ is even, $p$ divides $x$ and $p\equiv 1 \pmod{8}$. 
His proof used a result  established  in 1897 by C.~Störmer  
\cite{JFM28.0192.02} 
on the solution of the Pell--Fermat equation. 



The fundamental theorem of Siegel 
 \cite{JFM56.0180.05} 
in 1929 on the finiteness of  integer points on curves of genus $\ge 1$ implies that for each fixed pair of integers $(m,n)$ both $\ge 2$, the equation (\ref{E:Catalan}) has only finitely many solutions. The result, based on an extension of Thue's argument improving Liouville's inequality in Diophantine approximation, does not yield any bound for the set of solutions: such an effective result will be reached only after Baker's contribution. 

In 1932 S.~Selberg 
 proved that there is no solution to $x^4-1=y^n$ for $n\ge 2$ and positive $x$, $y$
  \cite{MR0249355,
 MR1259738}. 



In 1952 W.J.~LeVeque \cite{0047.04103}  proved that for fixed $a$ and $b$ both at least $2$ and $a\not=b$, the equation $a^x-b^y=1$ has  at most one solution $(x,y)$, unless $(a,b)=(3,2)$ where there are
two solutions $(1,1)$ and $(2,3)$.

The auxiliary results on the solutions to equation  (\ref{E:Catalan})  proved by 
J.W.S.~Cassels in 1953  \cite{MR0051851} and 1960 \cite{MR0114791}    include the fact   that  any solution to Catalan's equation $x^p-y^q=1$ with odd prime exponents $p$ and $q$  has the property that   $p$ divides $y$ and  $q$ divides $x$ (see also \cite{MR2067169}). 
They have been extremely useful in all subsequent works on Catalan's Problem.

In 1958 T.~Nagell  \cite{0083.03902}, showed that for $n=3$ there is only one solution to equation  (\ref{E:Catalan}).

In 1962, A.~Makowski and in 1964, S.~Hyyrö \cite{MR205925}  proved that there   is no  example of three consecutive perfect  powers in the sequence of perfect powers. 
Also in 
\cite{MR205925}  S.~Hyyrö gave upper bounds for the number of solutions $(x,y,n)$ of the Diophantine equation  $ax^n - by^n = z$.

In 1964, K.~Inkeri \cite{MR0168518} proved his first  criterion: {\it for a solution $(x,y,p,q)$ to Catalan's equation $x^p-y^q=1$ with $p$ and $q$ odd primes, 
\\
$\bullet$ if $p\equiv 3 \pmod 4$ then either $p^{q-1}\equiv 1 \pmod {q^2}$ or else the class number of $\bQ(\sqrt{p})$ is divisible by $q$.}
\\
He produced his second criterion only 25 years later  \cite{MR1042488}: 
\\
$\bullet$
{\it 
If $p\equiv 1 \pmod 4$ then either $p^{q-1}\equiv 1 \pmod {q^2}$ or else the class number of $\bQ(e^{2i\pi/ p})$ is divisible by $q$.
}

In 1965, Chao Ko  \cite{MR0183684}  proved that there is only one solution to equation  (\ref{E:Catalan})   with $m=2$. A proof is given in 
 Mordell's book \cite{MR0249355}.  
In 1976, a shorter proof  was given by Chein \cite{MR0404133}  using previous results due to C.~Störmer  (1898) \cite{JFM28.0192.02}  and T.~Nagell (1921).
 It is reproduced in Ribenboim's book
  \cite{MR1259738} \S~4.6. 
  A new proof of Chao Ko's Theorem is due to Mignotte \cite{MR2113081}.

In 1976,  R.~Tijdeman \cite{MR53:7941} proved that there are only finitely many solutions $(x,y,m,n)$  to equation  (\ref{E:Catalan}, and his proof yields an explicit (but extremely large) upper bound for the solutions.  
His solution involved a number of tools, some were due to earlier works on this question. Also he  used Baker's transcendence method which yields effective lower bounds for linear combinations with algebraic coefficients of logarithms of algebraic numbers. The estimates which were available at that time were not sharp enough, and the most part of Tijdeman's paper
\cite{MR53:7941}  
was devoted to a  sharpening of Baker's method. See also his paper  \cite{0319.10022}.

Tijdeman's proof is effective and yields an explicit upper bound for a possible exception to Catalan's Conjecture, but the first estimates 
\cite{0354.10008}
were too weak to be of practical use.

A survey on Catalan's Conjecture  \ref{C:Catalan} before 2000,  is due to M.~Mignotte
\cite{MR2002g:11034}, 
who contributed a lot  to the progress of the subject after the work by R.~Tijdeman: he paved the way for the final solution
\cite{%
MR1150825,
MR1250706,
MR1257210,
MR1336750,
MR1689515,
MR1816194}. 
 See also his survey papers \cite{MR2002g:11034,
MR2067169}. 
Among the works between 1976 and 2003 are also the following contributions by Yu.~Bilu, M.~Mignotte, Y.~Roy and G.~Hanrot: 
\cite{MR1387692,
MR1608717,
MR1440951,
MR1924488}.

 In 2003, following a previous  work of K.~Inkeri and M.~Mignotte,  P.~Mih{\u{a}}ilescu 
\cite{MR1968450} 
obtained a  class number free criterion for {C}atalan's conjecture: he showed that {\it for any solution to the equation $x^p-y^q=1$ with odd prime exponents $p$ and $q$, these exponents  must satisfy a double Wieferich condition
$$
p^{q-1}\equiv1 \pmod{q^2}\quad
\hbox{and}
\quad
q^{p-1}\equiv1 \pmod{p^2}.
$$
}
For a proof, see also \cite{MR2067169,MR1891421}

The next year, in 
\cite{MR2076124}, 
he completed the proof of Catalan's Conjecture \ref{C:Catalan}. 
This proof is also given in the Bourbaki Seminar by Yu.~Bilu in 2002 
\cite{MR2111637}  
who contributed to the tune up of the solution. 
In his original solution,  P.~Mih{\u{a}}ilescu used previous results (including  estimates coming from transcendental number theory as well as a result by F.~Thaine from 1988) together with further statements of his own on cyclotomic fields. Shortly after, thanks to the contributions of  Yu.~Bilu, Y.~Bugeaud, G.~Hanrot and P.~Mih{\u{a}}ilescu,  the transcendence part of the proof could be removed and a purely algebraic proof was produced  \cite{MR2152212}. 

Self contained  proofs of the solution of Catalan's Conjecture  \ref{C:Catalan} are given by T.~Mets{\"a}nkyl{\"a}  in 
\cite{MR2015449}, 
 H.~Cohen in 
 \cite{HenriCohenCatalan} 
and R.~Schoof in
 \cite{MR2459823}.

  An extension of Tijdeman's result on Catalan's equation (\ref{E:Catalan})   to the  equation $x^m-y^n=z^{\gcd(m,n)}$ has been done by A.J.~ van der Poorten in \cite{MR0450195}.  In  Theorem 12.4 of  \cite{MR88h:11002},   
T.N.~Shorey and R.~Tijdeman consider the equation  
              $$
              (x/v)^m-(y/w)^n=1,
              $$
where the unknowns  are $m>1$, $n>1$, $v$, $w$, $x$, $y$, with at least one of $v$, $w$, $x$, $y$ composed of fixed primes.

References to extensions of the  Catalan's Conjecture  \ref{C:Catalan}  to number fields are given in  Ribenboim's book 
  \cite{MR1259738}, 
Part X Appendix 1A. A more recent paper is due to R.~Balasubramanian  and  P.P.~Prakash \cite{BaluPandey}. 

\subsection{Perfect powers (continued)} \label{SS:PerfectPowers}

Several classes of exponential Diophantine equations of the form 
\begin{equation}\label{E:FermatCatalan}
Ax^p+By^q=Cz^r,
\end{equation}
 where $A$, $B$, $C$ are fixed positive integers, while the unknowns $x$, $y$, $z$, $p$, $q$, $r$ are non negative integers,  occur in the literature. Some of the integers $x$, $y$, $z$, $p$, $q$, $r$ may be assumed to be fixed, or to be composed of fixed primes.

\subsubsection{Beal's equation}

The case $A=B=C=1$ with $6$ unknowns is the so-called { \it Beal's  Equation}
\begin{equation}\label{E:Beal}
x^p+y^q=z^r;
\end{equation}
see
 \cite{BealCj},
 \cite{MR1714945}, 
  \cite{MR1761897} 
\S~9.2.D,
and
\cite{MR2074991}. 
Under the extra conditions 
\begin{equation}\label{Equation:Beal}
\frac{1}{ p}+\frac{1}{q}+\frac{1}{r}<1
\end{equation}
and $x$, $y$, $z$ are relatively prime,  only $10$
solutions (up to obvious symmetries) to  equation (\ref{E:Beal})  are known, namely
$$ 
1+2^3=3^2,
\qquad
2^5+7^2=3^4,
\qquad 7^3+13^2=2^9,
\qquad 
2^7+17^3=71^2,
$$
$$
3^5+11^4=122^2,
\qquad
17^7+76271^3=21063928^2, \qquad 
1414^3+2213459^2=65^7,
$$
$$
9262^3+15312283^2=113^7,
\qquad 
43^8+96222^3=30042907^2,
\qquad 
33^8+1549034^2=15613^3.
$$
Since the condition (\ref{Equation:Beal}) 
implies
$$
\frac{1}{ p}+\frac{1}{ q}+\frac{1}{ r}\le \frac{41}{ 42}\virgule
$$
the  $abc$ Conjecture \ref{Conj:abc} below  predicts  that the set of solutions to  (\ref{E:Beal})  is finite: this is the
{\it ``Fermat-Catalan'' Conjecture}  formulated by Darmon and Granville (\cite{MR1348707}; see also \cite{BealCj}). For all
known solutions, one at least of
$p$, $q$, $r$ is $2$; this led R.~Tijdeman,  D.~Zagier and A.~Beal  \cite{BealCj}  to
conjecture  that 
{\it there is no solution to Beal's equation (\ref{E:Beal})   with the further
restriction that each of $p$, $q$ and $r$ is $\ge 3$.  }

Darmon and Granville (\cite{MR1348707} proved that under the condition (\ref{Equation:Beal}), for fixed non--zero integers $A$, $B$, $C$, 
 the equation $Ax^p+By^q=Cz^r$ has only finitely many solutions $(x,y,z)\in\bZ^3$ with $\gcd(x,y,z)=1$. See also \cite{MR2216774,MR2465098}.

\subsubsection{Pillai's Diophantine equation}
 
Another subclass of equations (\ref{E:FermatCatalan}) is called {\it Pillai's Diophantine equation} in 
\cite{MR2242533}:  
this is the case where $q=0$ (and, say, $y=1$). 
With the notations of that paper, given positive integers $A$, $B$, $C$, $a$, $b$, $c$ with $a\ge 2$ and $b\ge 2$, the equation reads
\begin{equation}\label{E:PDE}
Aa^u-B b^v =c,
\end{equation}
where the unknown $u,v$ are non--negative integers. See also  Conjecture 5.26,  \S~5.7 of \cite{MR2465098}.

In 1971, using the transcendence method initiated by A.~Baker just a few years before, W.J.~Ellison 
\cite{MR0417051} 
(quoted in  \cite{MR961960} 
    p.~121--122) 
gave an effective refinement  of 
a result due to
 S.S.~Pillai  \cite{JFM57.0236.01}
by producing an explicit  lower bound for non--vanishing numbers of the form $|am^x-bn^y|$: {\it for any $\delta>0$ and any positive integers $a,b,m,n$, for  any sufficiently large $x$ (depending on $\delta,a,b,m,n$) the following estimate holds:
$$
|am^x-bn^y|\geq m^{(1-\delta)x}.
$$
}
In the special case $a=b=1$, $m=2$, $n=3$ (cf. also the book on Ramanujan  by G. H. Hardy \cite{MR0106147}),
he shows that 
$$
|2^x-3^y|>2^x e^{-x/10}
$$
 for $x\ge 12$ with $x\neq 13,14,16,19,27$, and all $y$. 

Chapter 12  ``Catalan equation and related equations'' of   \cite{MR88h:11002}  
is devoted to equations like (\ref{E:PDE}).

The result of T.N.~Shorey and R.~Tijdeman  \cite{MR0447110} in 1976  deals with  the equation 
$ax^m-by^n=k$, where $a$, $b$, $k$, $x$ are composed of fixed primes while  the unknowns are $y$, $m$, $n$ with $m>1$, $n>1$, $x>1$, $y>1$.
The  result of T.N.~Shorey,  A.J.~Van der Poorten,  R.~ Tijdeman and A.~Schinzel in  \cite{MR0472689} assumes
              $m$ fixed, $a$, $b$, $k$ are composed af fixed primes and the unknown are $n$, $x$, $y$.
              
A fundamental auxiliary  result  on the equation $wz^q=f(x,y)$ where $f$ is a polynomial with integer coefficients, $w$ is fixed or composed of fixed primes, and the unknown are $x$, $y$, $z$ and $q$, is also due to T.N.~Shorey, A.J.~Van der Poorten, R.~ Tijdeman and A.~Schinzel,  See 
also the {\it  commentary on A} by 
R.~Tijdeman, in Schinzel's Selecta 
\cite{MR2383194}.

After the work by  G.~P\'olya   \cite{JFM46.0240.04}  in 1918,  S.S.~Pillai \cite{JFM57.0236.01} in 1931  and  T.~Nagell  
\cite{0083.03902}   in 1958,
it is known that  equation (\ref{E:PDE}) has only finitely many solutions $(u,v)$. Under some conditions, the authors of  
 \cite{MR2242533} 
prove that there is at most one solution $(u,v)$. They apply also the $abc$ Conjecture to prove that the equation $a^{x_1}-a^{x_2}=b^{y_1}-b^{y_2}$ has only finitely many positive integer solutions $(a,b,x_1,x_2,y_1,y_2)$ under natural (necessary) conditions.

In 1986,  J.~Turk \cite{MR847290} gave an effective  estimate from  below for $|x^n-y^m|$, which was improved by B. Brindza, J.-H. Evertse and K. Gy\H{o}ry in 1991 and further refined  by Y.~Bugeaud in 1996
\cite{MR1396651} : 
{\it let $x$ be a positive integer and $y$, $n$, $m$ be integers which are $\ge 2$. Assume $x\sp{n}\not=y\sp{m}$. Then $|x\sp{n}-y\sp{m}|\ge m\sp{2/(5n)}n\sp{-5}2\sp{-6-42/n}$. }

 A related question is to estimate the maximum number of perfect powers in a comparatively small intervall. For instance the intervall   $[121,128]$ contains three perfect powers, namely $121=11^2$, $125=5^3$ and $128=2^7$. 

Turk's  estimate was improved the same year by J.H.~Loxton 
 \cite{MR87j:11071}, who 
studied perfect powers of integers in  intervalls $[N, N+N^{1/2}]$. An interesting feature of Loxton's  proof is the use of lower bounds for simultaneous linear combinations in logarithms. A gap in Loxton's paper was pointed out and corrected  in two different ways by D.~Bernstein 
\cite{0910.11057} and C.L.~Stewart \cite{MR2417459,1142.11009}. In \cite{MR2417459}, Stewart conjectures that there are infinitely many integers $N$ for which the interval $[N,N+N^{1/2}]$ contains three integers one of which is a square, one a cube and one a fifth power; further he conjectures that for sufficiently large $N$ the interval $[N,N+N^{1/2}]$  does not contain four distinct powers and if it contains three distinct powers then one of is a square, one iw a cube and the third is a fifth power;

Let  $a,b,k,x,y$  be  positive integers with $x>1$ and $y>1$. Upper bounds  for the number of integer solutions $(m,n)$ to the equation $ax\sp m-by\sp n=k$ in $m>1,n>1$ have been established by W.J.~LeVeque \cite{0047.04103} in 1952, 
T.N.~Shorey in 1986, Z.F.~Cao in 1990
and  Le Mao Hua
\cite{MR1168346} in 1992.

 In 1993, R.~Scott 
\cite{MR1225949}  considered the equation $p^x-b^y=c$, where $p$ is prime, and $b>1$ and $c$ are positive integers. In some cases he shows that there is at most one solution.

In 2001, M.~Bennett  \cite{MR1859761} studied the Diophantine equation
(\ref{E:Bennett})
 where $a\ge 2$, $b \ge 2$ and $c$ are given nonzero integers, the unknowns are  $x\ge 1$ and $y \ge 1$.   Improving an earlier result of   Le Mao Hua
\cite{MR1168346}, 
M.~Bennett  showed that equation  (\ref{E:Bennett})  has at most two solutions. 


In 2003, M.~Bennett  \cite{MR1955415}, 
proved that if $N$ and $c$ are positive integers with $N\ge 2$, then the Diophantine equation $$|(N+1)^x-N^y|=c$$ has at most one solution in positive integers $x$ and $y$ unless $(N,c)\in{(2,1),(2,5),(2,7),(2,13),(2,23), (3,13)}$. All positive integer solutions $x, y$ to the above equation are given in the exceptional cases. The proof is a combination of the following two ingredients. On the one hand, Bennett  uses an older work of his own on   the distance to the nearest integer which is denoted by $\Vert\cdot\Vert$. 
In 1993  \cite{MR1230126}, 
he sharpened earlier results of F.~Beukers (1981)
 and D. Easton  (1986)
by proving    $\Vert \bigl( (N+1)/ N\bigr)^k\Vert>3^{-k}$ for $N$ and $k$ positive integer with $4\leq N\leq k3^k$.
Further in \cite{MR1955415}, M.~Bennett established an effective result on  $\Vert(3/2)^k\Vert$ which is weaker than Mahler's noneffective estimate (see \S~\ref{SS:Waring}).  In \cite{MR2332068}  W.~Zudilin proved 
$$ \|(\tfrac{3}{2})^k\|>C^k=0.5803^k \qquad \text{for}\quad k\ge K, $$ where $K$ is an absolute  effective constant. This improves a previous estimate due to L. Habsieger  \cite{MR1957111} with $C=0.5770$. 


In 2004, R.~Scott and R.~Styer
\cite{MR2040155}  considered the equation 
\begin{equation}\label{E:ScottStyer}
p^x-q^y=c
\end{equation}
where  $p$, $q$ and $c$ are fixed, $p$ and $q$  are distinct primes and $c$ is a positive integer, while the unknowns are  $x$ and $y$ positive integers.  
Using sharp measures of linear independence for logarithms of algebraic numbers, they  show that if $q\not\equiv 1\pmod{12}$, then  (\ref{E:ScottStyer}) has at most one solution, unless either 
$$
(p,q,c)\in\{(3,2,1),(2,3,5),(2,3,13),(2,5,3),(13,3,10)\}
$$ 
or 
$$
q^{{\rm ord}_p q}\equiv 1\pmod{p^2},  \quad  {\rm ord}_p q  \quad \text{is odd and} \quad  {\rm ord}_p q >1.
$$ 
Further, the authors resolve a question posed by H. Edgar, viz. that  (\ref{E:ScottStyer})  with $c=2^h$ has at most one solution in positive integers ($x,y$) except when $x=3,q=2$ and $h=0$.  
 
 In 2006, the same authors 
\cite{MR2225282} investigated  the Diophantine equation 
$$ 
(-1)^ua^x+(-1)^vb^y=c 
$$ 
where $a$, $b$ and $c$ are fixed positive integers (with $a,b \geq2$). They show that this equation has at most two solutions in integers $(x,y,u,v)$ with $u,v\in\{0,1\}$, with certain (precisely determined) exceptions. This generalizes prior work by
M.~Bennett 
\cite{MR1859761}.  

In 1982, J.~Silverman \cite{MR664038} studied Pillai's equation $ax^m+by^n=c$ over function fields of projective varieties
(see also p.~220 in  \cite{MR88h:11002} 
and Appendix 1C p.~318 in   \cite{MR1259738}). 
In another work in 1987  \cite{MR895285},  J.~Silverman produces an estimate for the number of integral points on the twisted Catalan curve $y^m=x^n+a$ $(m\geq 2$, $n\geq 3$, $a\in\bZ)$, as a consequence of new 
 results he obtained on Lang's conjecture relating the number of integral points and the rank of the Mordell-Weil group of an elliptic curve.

\subsubsection{Fermat--Catalan Equation}\label{SSS:FermatCatalanEquation}

The {\it Fermat--Catalan equation} is another class of equations of the form  (\ref{E:FermatCatalan}), where the exponents $p$, $q$ and $r$ are equal. It reads
$$
ax^n+by^n=cz^n,
$$
where $a$, $b$, $c$ and $n$ are fixed positive integers with $n\ge 2$ while the unknowns $x$, $y$ and $z$ are positive integers. After the proof by A.~Wiles of Fermat's Last Theorem ($a=b=c=1$, $n\ge 3$, there is no solutions), a number of papers have extended his method (A.~Wiles, H.~Darmon, L.~Merel, A.~Kraus, A.~Granville, D.~Goldfeld and others).

\subsubsection{The Nagell--Ljunggren Equation}\label{SSS:NagelLjunggrenEquation}

The  equation
$$
\frac{x^n-1}{x-1}= y^q\ 
$$
with the unknowns $x>1$, $y>1$,   $n>2$, $q\ge 2$ has been investigated by T.~Nagell in 1020 and W.~Ljunggren in 1943. We refer to \cite{BiluBugeaudMignotte} Chap.~3 ({\it Two equations closely related to Catalan's)}. Known solutions are
$$
\frac{3^5-1}{3-1}=11^2,\quad
\frac{7^4-1}{7-1}=20^2,\quad
\frac{18^3-1}{18-1}=7^3.
$$

\subsubsection{The Goormaghtigh  Equation}\label{SSS:GoormaghtighEquation}

Equations of the form   
$$
a \frac{x^m-1}{x-1}= b  \frac{y^n-1}{y-1} \cdotp
$$
with fixed $a$ and $b$ and unknown $x$, $y$, $m$ and $n$ have also been investigated by many an author (see for instance 
  \cite{MR591599}), 
   Theorems 12.5, 12.6, 12.7 and 12.8
  in \cite{MR88h:11002}  
as well as p.~110 of     \cite{MR1259738}. 
One application is  the question raised by Goormaghtigh to determine the set  of integers having a single digit in two multiplicatively independent bases.  The only solutions 
$(m,n,x,y)$ of 
$$
x^{m-1}+x^{m-2}+\cdots+x+1=y^{n-1}+y^{n-2}+\cdots+y+1
$$
in integers $m>n>2$, $y>x>1$ are 
$(5,3,2,5)$ and $(13,3,2,90)$, corresponding to the development of $31$ in bases $2$ and $5$ for the first one, of $8191$ in bases $2$ and $90$ for the second one. A survey of such questions has been written by T.N.~Shorey
\cite{MR1925655}.

\subsubsection{Further special cases of  (\ref{E:FermatCatalan})}\label{SSS:FurtherSpecialCases}

Plenty of further results related to  the equation  (\ref{E:FermatCatalan}) are discussed in \cite{BiluBugeaudMignotte} Chap.~I ({\it Pillai's Equation}), including the {\it  generalized Ramanujan--Nagell equation }
$$
x^2+D=kp^n
$$
where $D$, $k$ and $p$ are fixed positive integers with $p$ prime, the unknowns are the positive integers $x$ and $n$ as well as  the {\it Lebesgue--Nagell type equations}
$$
D_1x^2+D_2=by^n
$$
where $D_1$, $D_2$, $b$ are fixed positive integers and the unknown are the positive integers $x$, $y$ and  $n$. We quote Problem 3 of \S~1.11 in  \cite{BiluBugeaudMignotte}:

\begin{quote}
{
Prove that there exists a positive constant $k$ for which there are only finitely many integers $n$ such that $n, n+2,\dots,n+2k$ are perfect powers.
 }
   \end{quote}
   
   Pillai's Conjecture on $x^p-y^q=2$ suggests that  $k=2$ should be a solution.

\section{Quantitative refinement of Pillai's Conjecture}

Here is a quantitative refinement of Pillai's Conjecture \ref{Conj:Pillai}. 

\begin{conjecture}\label{Conj:PillaiEffectif}
For any $\varepsilon>0$, there exists a constant $\kappa(\varepsilon)>0$ such that, for any positive integers $(a,\, b,\, x,\, y)$, with $x\ge 2$, $y\ge 2$ and $a^x\not=b^y$,
$$
|a^x-b^y|\ge \kappa(\varepsilon) \max\bigl\{  a^x\; ,\; b^y\bigr\}^{1-(1/x)-(1/y)-\varepsilon}.
$$
\end{conjecture}

A very special case of Conjecture \ref{Conj:PillaiEffectif}  is $x=3$, $y=2$: the conjecture reduces to the statement that for any $\epsilon>0$, there are only finitely many $(a,b)$ in $\bZ_{>0}\times\bZ_{>0}$ 
such that 
$$
0<|a^3-b^2|\le \max\{a^3,b^2\}^{(1/6)-\epsilon}.
$$
A stronger version of this statement  was proposed by M.~Hall  in  \cite{MR0323705}, where he removed the $\epsilon$:

\begin{conjecture}[Hall's Conjecture]\label{Conj:Hall}
There exists an absolute constant $C>0$ such that, for any pair $(x,y)$ of positive integers satisfying $x^3\not=y^2$, 
$$
|x^3-y^2|>C \max\{x^3,y^2\}^{1/6}.
$$
\end{conjecture}

References to the early history of Hall's Conjecture are given in \cite{MR0186620}. A review is given by N.~Elkies in \cite{MR1850598} \S~4.2. The fact that the value for $C$ in Conjecture \ref{Conj:Hall} cannot be  larger than $5^{-5/2}\cdot 54=.96598\dots$ follows from an example found by L.V.~Danilov \cite{MR767225}.

The tendency nowadays
(see for instance \cite{MR1850598}) is to believe only that a weaker form of this conjecture is  likely to be true, with the exponent $1/6$ replaced by $(1/6)-\epsilon$ (and a constant $C>0$ depending on $\epsilon$). Even an unspecified but constant value for the exponent would have dramatic consequences: 

\begin{conjecture}[Weak Hall's Conjecture]\label{Conj:WeakHall}
There exists an absolute constant $\kappa>0$ such that, for any pair $(x,y)$ of positive integers such that $x^3\not=y^2$, 
$$
|x^3-y^2|> \max\{x^3,y^2\}^{\kappa}.
$$
\end{conjecture}

In  \cite{MR517740},  Mohan Nair  shows  that the weak Hall's Conjecture \ref{Conj:WeakHall}  implies Pillai's Conjecture \ref{Conj:Pillai}.

In the direction of Hall's Conjecture we  mention a partial result due to V.G.~Sprind{\v{z}}uk 
\cite{MR1288309}:
{\it  there exists a positive absolute constant $c>0$ such that for any $(x,y)$ where $x$ and $y$ are integers satisfying $x\ge 6$ and $x^3\not=y^2$, 
$$
|x^3-y^2|\ge c\frac{\log x}{(\log\log x)^6}\cdotp
$$ 
}

The statement  \ref{Conj:PillaiEffectif} occurred initially in the book of Lang 
\cite{MR81b:10009}, introduction to Chapters X and XI, 
where heuristics were suggested for such lower bounds. It was later pointed out by A.~Baker \cite{BakerPise} that these heuristic may need some grain of salt, but it is remarkable that Conjecture 
\ref{Conj:PillaiEffectif} also follows as a consequence of the $abc$ Conjecture \ref{Conj:abc} of {\OE}sterlé 
and Masser  below.
Such occurrence may not be considered as independent: links between  the $abc$ Conjecture \ref{Conj:abc}  and measures of linear independence for logarithms of algebraic numbers have been pointed out by A.~Baker \cite{MR99e:11101,MR2107944} and P.~Philippon 
\cite{MR1680846}.

Baker's method yields explicit non--trivial lower bounds for the distance between distinct numbers of the form $\alpha_1^{\beta_1}\cdots\alpha_n^{\beta_n}$ when the $\alpha$'s and $\beta$'s are algebraic numbers. In the particular case where the exponents $\beta$'s are rational integers, it produces inequalities which may be applied to  Pillai's equation (\ref{E:Bennett}). See the appendix to \cite{MR84i:10014} 
Combining linear independence measures for logarithms with continued fraction expansions and brute force computation, P.L~Cijsouw, A.~Korlaar and R.~Tijdeman found all $21$ solutions $(x,y,m,n)$  to the inequality
$$
|x^m-y^n|<x^{m/2},
$$
where $m$ and $n$ are prime numbers and $x<y<20$.

Sharp enough estimates, like the ones which are conjectured in 
\cite{MR81b:10009},
would immediately solve Pillai's Conjecture \ref{Conj:Pillai}. We are far from this stage. However we can apply existing estimates and deduce the following  result  (see 
\cite{MR91h:11059})  
which  is quoted in   \cite{MR1259738} 
(C10.5):

\begin{proposition}\label{P:C10.5}
Let $W=1.37\times 10^{12}$. Assume $k\ge 1$, $m$, $n$, $x\ge 1$, $y\ge 2$, and that $x^m>k^2$. If $0<x^m-y^n\le k$, then $m<W\log y$ and $n<W\log x$.
\end{proposition}

We take this opportunity to correct an inaccuracy in the statement of the main theorem in 
\cite{MR91h:11059}, 
thanks to  a comment by Dong Ping Ping. 

\begin{theorem}\label{T:DongPingPing}
Let $n\ge 2$ be an integer, $a_1,\ldots,a_n$ multiplicatively independent positive integers, $b_1,\ldots,b_n$ rational integers which are not all zero, so that 
$$
a_1^{b_1}\cdots a_n^{b_n}\not=1.
$$
Set 
$C_0(n)=e^{3n+9} n^{4n+5}$ and
\begin{equation}\label{E:DongPingPing}
M=\max\bigl\{ n^{4n} e^{20n+10},
\; 
\max_{1\le i\not=j\le n}  |b_i|/\log a_j\bigr\}.
\end{equation}
Then
$$
\bigl|a_1^{b_1}\cdots a_n^{b_n} - 1\bigr| >\exp\bigr\{-C_0(n) \log M \log a_1\cdots\log a_n\bigr\}.
$$

\end{theorem}

The difference with \cite{MR91h:11059}
is that (\ref{E:DongPingPing}) was replaced by 
$$
M=\max\bigl\{ n^{4n} e^{20n+10},\; |b_n|/\log a_0\bigr\},
$$
where the ordering of the numbers $a_i^{b_i}$ was selected  so that 
$$
a_n^{|b_n|}=\max_{1\le i\le n} a_i^{|b_i|}
\quad\hbox{and}\quad
a_0=\min_{1\le i\le n-1} a_i.
$$

\subsection{The $abc$ Conjecture}\label{SS:abc}

For a positive integer \m{n}, we
denote by
\M{
R(n)=\prod_{p|n}p
}
the 
{\it   radical\/} or 
{\it   square free part\/} of \m{n}.

The
$abc$ Conjecture resulted from a discussion between D.~W.~Masser
and J.~{\OE}sterl\'e
 --- see \cite{MR1065152}  
and  \cite{MR992208},  
p.~169; 
see also  
\cite{MR90k:11032}, 
\cite{MR93a:11048} 
Chap.~II \S~1,
\cite{MR1764797}, 
\cite{MR1761897} \S~9.4.E, 
 \cite{MR1771576}, 
 \cite{MR1740351}, 
\cite{MR1878556} Ch.~IV  \S~7, 
 \cite{MR2076335} B19,  
\cite{MR2074991} 
\S~2.1 and the $abc$ Conjecture Home Page \cite{NitajABC} created in 2002 and maintained by A.~Nitaj.  

The reference \cite{NitajABC} provides information on a large variety of consequences of the $abc$ Conjecture \ref{Conj:abc}, including Fermat's Last Theorem, the  Fermat--Catalan  equation (see \S~\ref{SSS:FermatCatalanEquation}), Wieferich primes, the Erd\H{o}s--Woods Conjecture in logic, arithmetic progressions with the same prime divisors, Hall's Conjecture \ref{Conj:WeakHall}, the Erd\H{o}s--Mollin--Walsh Conjecture on consecutive powerful numbers, Brown Numbers and Brocard's Problem, Szpiro's Conjecture for elliptic curves, Mordell's conjecture in Diophantine geometry \cite{MR1141316}, squarefree values of polynomials, Roth's theorem in Diophantine approximation,   Dressler's Conjecture   (between any two different positive integers having the same prime factors there is a prime), Siegel zeros of Dirichlet $L$--functions, Power free--values of polynomials, squarefree--values of polynomials, bounds for the order of the Tate--Shafarevich group of an elliptic curve, Vojta's height conjecture for   algebraic points of bounded degree on a curve, the powerful part of terms in binary recurrence sequences, Greenberg's Conjecture on the vanishing of the Iwasawa invariants of the cyclotomic $\bZ_p$--extension of a totally real number field, 
exponents of class groups of quadratic fields, limit points of the set of $\log(c)/\log(R(abc)$ for relatively prime $a$, $b$, $c$ with $a+b=c$, fundamental units of certain quadratic and biquadratic fields, the Schinzel--Tijdeman Conjecture on the Diophantine equation $P(x)=y^2z^3$, Lang's conjecture $h(P)>C(K)\log |N_{K/\bQ}D_{E/K}$ relating the minimal height of a non--torsion point on an elliptic curve $E$ and the minimal discriminant of $E$ over the number field $K$, Lang's Conjecture bounding  the number of integral points over a number field $K$  in terms of the rank of an elliptic curve over  $K$, rounding reciprocal square roots, the Diophantine equation $p^v-p^w=q^x-q^y$, the number of quadratic fields generated by a polynomial.

\begin{conjecture} [$abc$ Conjecture]\label{Conj:abc}
For each \m{\varepsilon>0},
there exists  $\kappa(\varepsilon)>0$  such that, if \m{a}, \m{b} and \m{c} in $\bZ_{>0}$ are relatively prime and satisfy \m{a+b=c}, 
 then
\M{
c<\kappa(\varepsilon)R(abc)^{1+\varepsilon}.
}
\end{conjecture}

Write $R$ for $R(abc)$. In 1986,  C.L.~Stewart and R.~Tijdeman \cite{MR0863221}
proved the existence of an absolute constant $\kappa $ such that, under the assumption of the $abc$ Conjecture \ref{Conj:abc}, 
\begin{equation}\label{E:StewartYu}
\log  c <  \kappa  R^{15 }.
\end{equation}
The proof involved a $p$--adic measure of  linear independence for  logarithms of algebraic numbers, due to van der Poorten. This measure was  improved by Yu Kunrui, and as a consequence,  in 1991,    C.L.~Stewart and Yu Kunrui  \cite{MR1129361}  improved the upper bound to
  $$
\log  c <  \kappa(\varepsilon)  R^{2/3+\varepsilon },
$$
where $\kappa(\varepsilon)$ is a positive effectivley computable number depending only of $\varepsilon$. In fact their result is more precise 
(cf.~ \cite{MR99a:11088b}, \S~2.6) 
$$
 \log c <    R^{(2/3)+ ( \kappa/ \log\log R)} ,
$$
where $\kappa$ is now  an absolute constant.

A completely explicit estimate   has been worked out by Wong Chi Ho in his master's thesis [Hong Kong Univ. Sci. Tech., Hong Kong, 1999]: For $c>2$, under the assumptions of the $abc$ Conjecture \ref{Conj:abc},  the estimate 
$$
\log c\le R\sp{(1/3)+(15/\log\log R)}
$$ 
holds.

In 2001,  C.L.~Stewart and Yu Kunrui   \cite{MR1831823} achieved the asymptotically stronger estimate 
$$
\log c< \kappa \sb{3}R\sp{1/3}(\log R)\sp{3}
$$ 
with an effectively computable positive absolute constant $\kappa$. 

The main new tool which enables them to refine their previous result is again a $p$--adic linear independence measure for logarithms of algebraic numbers, due to Yu Kunrui,
 which is an ultrametric analog of an Archimedean measure due to E. M. Matveev.

Another estimate in the direction of the $abc$ Conjecture \ref{Conj:abc} is also proved in the same paper  \cite{MR1831823} :
{\it  If $a$, $b$, $c$ are relatively prime positive integers with $a+b=c$ and $c>2$, then }
$$
\log c< p' R\sp{ \kappa (\log\log\log R\sp{*})/\log\log R},
$$ 
where $R\sp{*}=\max\{R,16\}$, $p'=\min\bigl\{ p\sb{a}, p\sb{b}, p\sb{c}\bigr\} $ and $p\sb{a}, p\sb{b}, p\sb{c}$ denote the greatest prime factors of $a$, $b$, $c$ respectively, with the convention that the greatest prime factor of $1$ is $1$.

An extension of the $abc$ Conjecture \ref{Conj:abc}  to number fields has been investigated by  K.~Gy\H{o}ry  \cite{GyoryABC}. Extensions to the equation $x_0+\cdots+x_n=0$ is proposed in \cite{MR2465098} \S~5.7.



Let us show that the $abc$ Conjecture \ref{Conj:abc}  implies the quantitative refinement of Pillai's Conjecture \ref{Conj:Pillai}  which was stated in Conjecture  \ref{Conj:PillaiEffectif}. By symmetry, we may assume $a^x>b^y$. Define $\delta=a^x-b^y$. Let $d$ be the gcd of $a^x$ and $b^y$. We apply the $abc$ Conjecture \ref{Conj:abc} to the numbers $A=b^y/d$, $B=\delta/d$, $C=a^x/d$, which are relatively prime and satisfy $A+B=C$. The radical of $ABC$ is bounded by $\delta ab/d$. Hence 
$$
a^x\le \kappa(\epsilon) \delta^{1+\epsilon} a^{1+\epsilon}b^{1+\epsilon}.
$$
From $b\le a^{x/y}$ we deduce the estimate from Conjecture  \ref{Conj:PillaiEffectif}.

\subsection{Connection with Waring's Problem}\label{SS:Waring}

For each integer \m{k\ge 2}, denote by $g(k)$ the smallest integer $g$ such that any positive integer is the sum of at most $g$ integers of the form $x^k$.  
 
It is easy to prove a lower bound for $g(k)$, namely  \m{g(k)\ge I(k)} with $I(k)=2^k+[(3/2)^k]-2$. 
Indeed, write 
\M{
3^k=2^kq+r
\quad
 \hbox{with }\quad 
0<r<2^k,\quad q=[(3/2)^k], \quad \hbox{so that }\quad I(k)=2^k+q-2,
}
and  consider the integer
\M{
N=2^kq-1=(q-1)2^k+ (2^k-1)1^k .
}
Since \m{N<3^k}, writing \m{N} as a sum of \m{k}-th powers can involve no term
\m{3^k},
 and since \m{N<2^kq}, it involves at most \m{(q-1)} terms \m{2^k}, all
others being \m{1^k}; 
hence it requires a total number of at least
\m{(q-1)+ (2^k-1)=I(k)} terms.

 The {\it  ideal Waring's Theorem} is the following conjecture, dating back to
1853: 

\begin{conjecture}\label{Conj:IdealWaringTheorem}
For any $k\ge 2$, the equality  $g(k)=I(k)$  holds.

\end{conjecture}

This conjecture has a long and interesting story. For the small values of $k$, one of the most difficult case to settle was $g(4)=19$, which was eventually solved by R.~Balasubramanian, J.M.~Deshouillers and F.~Dress \cite{0594.10040,0594.10039}.
The difficulty with the exponent $4$ arises from  the fact that the  difference between $g(k)$ and the other Waring exponent $G(k)$ is comparatively small for $k=4$; here $G(k)$ is  the smallest integer $g$ such that any {\it sufficiently large} positive integer is the sum of at most $g$ integers of the form $x^k$.  Indeed $G(4)=16$: the upper bound $G(4)\le 16$ is a result of H.~Davenport in 1939, while Kempner proved that no integer of the form $16^n\times 31$ can be represented as a sum of less than $16$ fourth powers.

 L.E.~Dickson and S.S.~Pillai 
 (see for instance \cite{HW}  
 Chap.~XXI
 or \cite{MR961960}  
 p.~226 Chap.~IV) 
 proved independently in 1936 that the ideal Waring's Theorem
  holds provided that the remainder  \m{r=3^k-2^kq} satisfies
\begin{equation}\label{E:DicksonPillaiWaring}
r\le 2^k-q-3.
\end{equation}
The condition
 (\ref{E:DicksonPillaiWaring}) 
 is satisfied for \m{3\le k\le 471~600~000}, as well as for sufficiently large $k$, as shown by 
K.~Mahler \cite{MR0093509} 
in 1957:   by means of Ridout's extension of the Thue-Siegel-Roth theorem, he proved that if  $\alpha$ is any positive algebraic number,  $u$ and $v$ are relatively prime integers satisfying $u>v\geq 2$ and  $\varepsilon>0$, then the inequality 
$$ \Vert \alpha(u/v)^n\Vert <e^{-\varepsilon n} $$ 
is satisfied by at most a finite number of positive integers $n$.
 
From the  special case $\alpha=1$, $u=3$, $v=2$ and $0<\varepsilon<\log{ (4/3)}$, it follows that any sufficiently large integer $k$ satisfies 
 (\ref{E:DicksonPillaiWaring}), hence 
the ideal Waring's Theorem holds except possibly for a finite number of values of $k$.


The following result is due to 
Sinnou David (personal communication, unpublished).

\begin{proposition}\label{P:Sinnou}
Assume that there exist two positive constants $\theta$ and $\kappa$ with 
$$
1<\theta<\frac{\log 3}{2\log (3/2)}
$$
having the following property:   for any triple $(a,b,c) $ of positive relatively prime integers  satisfying \m{a+b=c}, 
the inequality
\M{
c<\kappa(\varepsilon)R(abc)^{\theta}
}
holds. Set
$$
k_0=
\left[
\frac{2\theta\log 6+\log \kappa}{\log 3-2\theta\log(3/2)}
\right]
.
$$
Then for any positive integer $k\ge k_0$,  the quotient $q$ and the remainder $r$ of the division of $3^k$ by $2^k$, namely
$$
3^k=2^kq+r, \quad  0<r<2^k,
$$
satisfy  (\ref{E:DicksonPillaiWaring}).

\end{proposition}

We have seen 
that any integer $k$ for which the inequality  (\ref{E:DicksonPillaiWaring})  holds has $g(k)=I(k)$. 
Hence the ideal Waring Theorem would follow from an explicit solution to the \m{abc} Conjecture.

\begin{proof}[Proof of Proposition  \ref{P:Sinnou}]
Under the assumptions of Proposition \ref{P:Sinnou}, assume $k$ is a positive integer   for which $r\ge 2^k-q-2$. We want to bound $k$ by 
$$
k\bigl( \log 3-2\theta\log(3/2)\bigr) \le \theta\log 36+\log \kappa.
$$
 Denote by $3^\nu$ the gcd of $3^k$ and $2^k(q+1)$ and set
$$
a=3^{k-\nu},\quad c=3^{-\nu} 2^k(q+1),\quad b=c-a.
$$
Then $a$, $b$, $c$ are relatively prime positive integers, they satisfy $\gcd(a,b,c)=1$ and 
$$
b=3^{-\nu} (2^k-r)\le3^{-\nu}(q+2).
$$
The radical $R$ of the product $abc$ is bounded by 
$$
R\le 
6(q+1)b\le 3^{-\nu} 6(q+1)(q+2).
$$
For $q\ge 1$ we have    $(q+1)(q+2)< 6q^2 $. 
The assumption of Proposition \ref{P:Sinnou} yields 
$$
 3^{-\nu} 2^k(q+1) <\kappa \bigl( 3^{-\nu} 36 q^2\bigr)^\theta.
 $$
 Since $\nu\ge 0$ and $\theta>1$, we derive
 $$
 2^k< 6^{2\theta}  \kappa q^{2\theta-1}.
 $$
From $r>0$ we deduce  $3^k>2^kq$, hence
$$
\bigl(2 (2/3)^{2\theta-1}\bigr)^k<  6^{2\theta} \kappa.
$$
\end{proof}

\def\cprime{$'$} \def\cprime{$'$}

\vfill

\end{document}